\documentclass{amsart}

\usepackage{amssymb}

\newtheorem{thm}{Theorem}[section]

\newtheorem{lem}[thm]{Lemma}

\theoremstyle{definition}

\theoremstyle{remark}

\numberwithin{equation}{section}

\begin{document}
\title{On a damped nonlinear beam equation}
\author{David Raske}
\email{nonlinear.problem.solver@gmail.com}

\begin{abstract}
In this note we analyze the large time behavior of solutions to an initial/boundary problem involving a damped nonlinear beam equation. We show that under physically realistic conditions on the nonlinear terms in the equation of motion the energy is a decreasing function of time and solutions converge to a stationary solution with respect to a desirable norm.
\end{abstract}                        

\maketitle

\section{Introduction}
Let $a$, $b$ be two real numbers. Let $T$ be a positive real number. Let $F_1$ and $F_2$ be two functions from $\mathbb{R}$ into $\mathbb{R}$. Let $f$ be a map from $[0,T]$ into $L^2(a,b)$
Initial/boundary value problems of the form
\begin{equation}\label{eq 1.2}
\begin{split}
& m u_{tt} + \sigma u_{xxxx} + F_1(u_t) + F_2(u) = f(x,t), \text{ on } (a,b) \times (0,T) \\
& u(a,t)=0=u(b,t),  \text{ for all } t \text{ in } (0,T) \\
&  u_{xx}(a,t)=0=u_{xx}(b,t),  \text{ for all } t \text{ in } (0,T) \\
& u(x,0)=u_0(x), u_t(x,0) = u_1(x) \text{ for all } x \text{ in } (a,b). \\
\end{split}
\end{equation}
arise naturally in the study of vibrations in slender beams. In particular, they have been used to model vibrations in suspension bridges and railroad tracks. A treatment of the existence and uniqueness problem of the above can be found in \cite{R}.

Before we continue we need to define some spaces.  Let $\Omega := (a,b)$, and let $H^2_*(\Omega)$ be the intersection of $H^1_0(\Omega)$ with $H^2(\Omega)$. It is equipped with the inner-product  
\begin{equation}
(u,v)_{H^2_*} = \int_\Omega u_{xx} v_{xx} dx.
\end{equation}
Let $H^4_*(\Omega)$ be the elements $u$ of $H^4(\Omega)$ such that $u \in H^2_*(\Omega)$ and $u_{xx} \in H^2_*(\Omega)$; it is equipped with the inner-product 
\begin{equation}
(u,v)_{H^4_*} = \int_\Omega u_{xxxx} v_{xxxx} dx.
\end{equation}
 In \cite{R} it is shown that the above inner-products make $H^2_*(\Omega)$ and $H^4_*(\Omega)$ Hilbert spaces.

Now, suppose there exists a function $u: [0,T] \rightarrow L^2((a,b))$ such that $u \in C^0( [0,T] ,H^4_0((a,b)) )$, $u \in C^1( [0,T] ,H^2_0((a,b)) )$, $u \in C^2( [0,T] ,L^2((a,b)) )$, $u(0)=u_0$, $u'(0)=u_1$, and such that for all $\phi \in C^\infty_c(0,T)$ and $v \in L^2((a,b))$ we have
\begin{equation}
\begin{split}
&\int_0^T m (u'',v)_{L^2}\phi(t) \, dt \\
& = -\int_0^T  [\sigma (u_{xxxx},v)_{L^2} + (F_1(u'),v)_{L^2} + (F_2(u),v)_{L^2} -( f(x,t),v)_{L^2}]\phi(t) \, dt. 
\end{split}
\end{equation}
We will then call $u$ a pseudo classical solution of \eqref{eq 1.2}. 

Now, suppose there exists a function $u: [0,\infty) \rightarrow L^2((a,b))$ such that $u \in C^0( [0,\infty) ,H^4_0((a,b)) )$, $u \in C^1( [0,\infty) ,H^2_0((a,b)) )$, $u \in C^2( [0,\infty) ,L^2((a,b)) )$, $u(0)=u_0$, $u'(0)=u_1$, and such that for all $\phi \in C^\infty_c(0,\infty)$ and $v \in L^2((a,b))$ we have
\begin{equation}
\begin{split}
&\int_0^\infty m (u'',v)_{L^2}\phi(t) \, dt \\
& = -\int_0^\infty  [\sigma (u_{xxxx},v)_{L^2} + (F_1(u'),v)_{L^2} + (F_2(u),v)_{L^2} -( f(x,t),v)_{L^2}]\phi(t) \, dt. 
\end{split}
\end{equation}
We will then call $u$ a global pseudo classical solution of the initial/boundary value problem
\begin{equation}\label{pseudo}
\begin{split}
& m u_{tt} + \sigma u_{xxxx} + F_1(u_t) + F_2(u) = f(x,t), \text{ on } (a,b) \times (0,\infty) \\
& u(a,t)=0=u(b,t),  \text{ for all } t \text{ in } (0,\infty) \\
&  u_{xx}(a,t)=0=u_{xx}(b,t),  \text{ for all } t \text{ in } (0,\infty) \\
& u(x,0)=u_0(x), u_t(x,0) = u_1(x) \text{ for all } x \text{ in } (a,b). \\
\end{split}
\end{equation}

A natural question to ask if a global pseudo classical solution exists is whether or not it converges to a solution of the boundary value problem
\begin{equation} \label{eq 1.1}
\begin{split}
& \sigma u_{xxxx} + F_2(u) = f(x) \text{ on } (a,b),  \\
& u(a)=0=u(b), \\
& u_{xx}(a)=0=u_{xx}(b). \\
\end{split}
\end{equation}
as $t \rightarrow \infty$. Since \eqref{eq 1.2} is derived from a damped mechanical problem, it's also natural to ask whether or not the total energy is decreasing or not if $f$ is independent of time. These two questions are the focus of this paper, and the main result is 

\begin{thm} Let $a$ and $b$ be two real numbers, with $a < b$. Let $m$, $\sigma$, and $c$ be three positive real numbers. Let $d$ be a non-negative real number. Let $\Omega$ be the open interval $(a,b)$. Let $f$ be a $C^1$ map from $[0,\infty)$ into $L^2(\Omega)$ that is independent of $t \in [0,\infty)$. Let $F_1: \mathbb{R} \rightarrow \mathbb{R}$ be defined as follows: If we write $F_1=F_1(x)$, then $F_1(x) = cx + d|x|x$. Let $F_2: \mathbb{R} \rightarrow \mathbb{R}$ be a continuously differentiable function such that $F_2$ is non-decreasing and such that there exists a real number $r$ such that $F_2(r)=0$. Furthermore, let $u_0$ be an element of $H_*^4(\Omega)$ and let $u_1$ be an element of $H^2_*(\Omega)$. Suppose \eqref{pseudo} possesses a unique global pseudo classical solution, $u$. Then, this solution satisfies $u(t) \rightarrow \hat{u}$ in $H^2_*(\Omega)$ and $u'(t) \rightarrow 0$ in $L^2(\Omega)$, where $\hat{u}$ is the unique solution of the boundary value problem \eqref{eq 1.1}. 
\end{thm}

Here we restrict our attention to the nonlinear damping term, $F_1(u_t) = cu_t + d|u_t|u_t$, because it is thought to provide a realistic model of damping, but is simple enough to result in a tractable problem. The linear term dominates when $|u_t|$ is small, but $|F_1(u_t)|$ grows like $d(u_t)^2$ as $|u_t| \rightarrow \infty$. This seems to be more realistic than just $F_1(u_t) = cu_t$ or $F_1(u_t)=d |u_t|u_t$.

In section two of this paper we make some observations about how global pseudo classical solutions of \eqref{eq  1.2} evolve as $t \rightarrow \infty$, for general damping terms $F_1$. In section three of this paper we use these observations to prove Theorem 1.2. 

Finally, it should be noted that Theorem 1.2 holds with $F_1(x) = a x + b|x|^{p-2}x$, where $a$ and $b$ are a positive real numbers and where $p$ is a positive real number greater than or equal to $2$ in value. It should also be noted that the results of Theorem 1.1 hold for a much broader range of elliptic differential operators and boundary conditions that the ones used here. We restrict our attention to the one dimensional bi-Laplacian augmented with homogeneous Navier boundary conditions, because the author is interested in bending phenomena and does not want to distract attention away from the nonlinear terms in the equation of motion.

\section{Preliminaries}
The goal of this section is to make some useful observations about how solutions of \eqref{eq 1.2} evolve with respect to time, especially when $f$ is independent of $t$. So let $T$, $m$, and $\sigma$ be positive real numbers, and let $a$ and $b$ be two real numbers with $a < b$. Let $\Omega$ be the interval $(a,b)$. Let $f_2: \mathbb{R} \rightarrow \mathbb{R}$ be a $C^2$ function such that $f_2(t) = \int_{r}^t F_2(s) ds$, where the real number $r$ is defined above. Let $F_1: \mathbb{R} \rightarrow \mathbb{R}$ be a non-decreasing $C^1$ function such that $F_1(0)=0$. Let $f \in C^1([0,\infty); L^2(\Omega))$. Finally, let $u_0 \in H^4_*(\Omega)$ and let $u_1 \in H^2_*(\Omega)$. 

Now suppose that (1.1) has a unique pseudo classical solution $u_T$. It follows that for all $v \in H^2_*(\Omega)$ we have
\begin{equation} \label{eq 2.1}
\begin{split}
& m<u_T''(t),v> + (F_1(u_T'(t),v)_{L^2} + \sigma(u_T(t),v)_{H^2_*} + (F_2(u_T(t)),v)_{L^2} \\ & = (f(t),v)_{L^2},
\end{split}
\end{equation}
for all $t \in [0,T]$. This, in turn, allows us to write
\begin{equation} \label{eq 2.2}
\begin{split}
& m<u_T'(t),u_T'(t)> + (F_1(u_T'(t),u_T'(t))_{L^2} + \sigma(u_T(t),u_T'(t))_{H^2_*} + \\ & (F_2(u_T(t)),u_T'(t))_{L^2} = (f(t),u_T'(t))_{L^2},
\end{split}
\end{equation}
for almost every $t \in [0,T]$. Since $u_T(t)$ is a pseudo classical solution of \eqref{eq 1.2} all of the above terms are summable. It follows that we can integrate all of them from $0$ to $t$, where $0 \leq t \leq T$. This leaves us with
\begin{equation} \label{eq 2.3}
\begin{split}
& \int_0^t (m<u_T''(s),u_T'(s)> + (F_1(u_T'(s),u_T'(s))_{L^2} + \sigma(u_T(s),u_T'(s))_{H^2_*} + \\ & (F_2(u_T(s)),u'_T(s))_{L^2}) ds = \int_{0}^t (f(s),u_T'(s))_{L^2} ds,
\end{split}
\end{equation}
for all $t \in [0,T]$. It follows that
\begin{equation} \label{eq 2.4}
\begin{split}
& \int_0^t (m \frac{1}{2} \frac{d}{ds} (u_T'(s),u_T'(s))_{L^2} + (F_1(u_T'(s)),u_T'(s))_{L^2} + \sigma \frac{1}{2} \frac{d}{ds} (u_T(s),u_T(s))_{H^2_*} + \\ & (F_2(u_T(s)),u_T'(s))_{L^2}) ds = \int_0^t (f(s),u_T'(s))_{L^2} ds,
\end{split}
\end{equation}
for all $t \in [0,T]$. Let $V$ be a function from $H^2_*(\Omega)$ into the real numbers such that $V(u) = \int_\Omega f_2(u(x)) dx$ for all $u \in H^2_*(\Omega)$. Calculus then gives us that
\begin{equation} \label{eq 2.5}
\begin{split}
& \frac{m}{2} (u_T'(t),u_T'(t))_{L^2} + \int_0^t (F_1(u_T'(s)),u_T'(s))_{L^2} ds  + \frac{\sigma}{2} (u_T(t),u_T(t))_{H^2_*} \\ & + V(u_T(t)) = \int_{0}^t (f(s), u'_T(s))_{L^2} ds + \frac{m}{2} (u_T'(0),u_T'(0))_{L^2} +  \\ & \frac{\sigma}{2} (u_T(0),u_T(0))_{L^2} + V(u_T(0)),
\end{split}
\end{equation}
for all $t \in [0,T]$. 

Now suppose \eqref{pseudo} has a unique global pseudo classical solution, $u(t)$. It follows that a pseudo classical solution of \eqref{eq 1.2} exists for all $T  \in (0,\infty)$, and $u_T(t) = u(t)$ for all $t \in [0,T]$. This, in turn, allows us to conclude that
\begin{equation} \label{limited}
\begin{split}
& \frac{m}{2} (u'(t),u'(t))_{L^2} + \int_0^t (F_1(u'(s)),u'(s))_{L^2} ds  + \frac{\sigma}{2} (u(t),u(t))_{H^2_*} \\ & + V(u(t)) = \int_{0}^t (f(s), u'(s))_{L^2} ds + \frac{m}{2} (u'(0),u'(0))_{L^2} +  \\ & \frac{\sigma}{2} (u(0),u(0))_{L^2} + V(u(0)),
\end{split}
\end{equation}
for all $t \in (0,\infty)$. 

It is difficult to draw substantial conclusions from the above equation for general $f$, but if $f$ is independent of time, we can say much more. 
We will now assume that $f$ does not depend on $t$ for the remainder of the section. Now note that if $f$ is independent of $t$, \eqref{limited} implies that
\begin{equation} \label{eq 2.6a}
\begin{split}
& \frac{m}{2} (u'(t),u'(t))_{L^2} + \frac{\sigma}{2} (u(t),u(t))_{H^2_*} + V(u(t)) - (f,u(t))_{L^2} \\ & = -(f,u(0))_{L^2} + \frac{m}{2} (u'(0),u'(0))_{L^2} + \frac{\sigma}{2} (u(0),u(0))_{L^2} + V(u(0)) \\ & -  \int_0^t (F_1(u'(s)),u'(s))_{L^2} ds.
\end{split}
\end{equation}

Let us now turn our attention to finding two functionals that will be useful in the analysis of how solutions of \eqref{eq 1.2} evolve with respect to time if $f$ is independent of time. Since this problem is motivated by a problem in continuum mechanics, previous experience suggests that we should look at the potential energy and the total energy. The former is
\begin{equation} \label{static}
\Sigma_T(u) =  \frac{\sigma}{2} (u,u)_{H^2_*} + V(u) - (f,u)_{L^2}.
\end{equation}
The latter is 
\begin{equation} \label{eq 2.6b}
E_u(t) = \frac{m}{2} (u'(t),u'(t))_{L^2} + \frac{\sigma}{2} (u(t),u(t))_{H^2_*} + V(u(t)) - (f,u(t))_{L^2}.
\end{equation}
Notice that
\begin{equation} \label{properties}
\Sigma_T \text{ is coercive and bounded below on } H^2_*(\Omega).
\end{equation}
An immediate consequence of \eqref{eq 2.6a} is that 
\begin{equation} \label{decreasing}
E_u \text{ is non-increasing with respect to time,} 
\end{equation}
and, in particular, $E_u(t)$ is bounded from above on $[0,\infty)$. Combining this observation with \eqref{properties} gives us that
\begin{equation} \label{norm}
\text{There exists a constant } C \text{ such that } (u(t),u(t))_{H^2_*} \leq C \text{ for all } t \in (0,\infty).
\end{equation}
Now, combining \eqref{eq 2.6a} with \eqref{properties} we see that
\begin{equation} \label{eq 2.7}
\begin{split}
& \text{There exists a constant } C \text{ such that } \int_0^t (F_1(u'(s)),u'(s))_{L^2} ds \leq C  \\ & \text{ for all } t \in (0,\infty).
\end{split}
\end{equation}

Before we continue, let us make another observation about $\Sigma_T$. Note that if $f_2$ is convex with a derivative that vanishes somehere on the real line, we can use the direct method of the calculus of variations to show that there exists  a unique minimizer of $\Sigma_T$ over $H^2_*(\Omega)$. Furthermore, this minimizer will be the unique (weak) solution of the boundary value problem \eqref{eq 1.1}. Previous experience then suggests that if $f_2$ is a convex function whose derivative vanishes somewhere then solutions of \eqref{eq 1.2} should converge to the solution of  \eqref{eq 1.1} with respect to $H^2_*(\Omega)$ norm as time goes to infinity. This is the subject of the next section.

Finally, we prove a result that will be useful in the next section.

\begin{lem} Let $f: [0,\infty) \rightarrow \mathbb{R}$ be a continuous non-negative function. Suppose as well that there exists a positive real number $C$ such that $\int_0^t f(s) ds \leq C$ for all $t \in (0,\infty)$. Then there exists a sequence $\{t_n\}_{n=1}^\infty$ with $t_n \uparrow \infty$ such that $\int_{t_n}^{t_n+1} f(s) ds \rightarrow 0$ as $n \rightarrow \infty$, and 
\begin{equation}
\lim_{n \rightarrow \infty} [f(t_n)+f(t_n+1)] = 0.
\end{equation}

\end{lem}

\begin{proof} Since $f$ is non-negative and since there exists a positive real number $C$ such that $\int_{0}^t f(s) ds \leq C$ for all $t \in [0,\infty)$, it follows that $\int_{2n}^{2n+1} f(s) ds \rightarrow 0$ as $n \rightarrow \infty$. Calculus then gives us that $\int_{2n}^{2n+1} f(s+1) ds \rightarrow 0$ as $n \rightarrow \infty$. It follows that 
\begin{equation}
\int_{2n}^{2n+1} [f(s) + f(s+1)] ds \rightarrow 0 \text{ as } n \rightarrow \infty.
\end{equation}
This in turn allows us to use the mean value theorem for integrals to conclude that there exists a sequence $\{t_n\}_{n=1}^\infty$ with $t_n \in (2n,2n+1)$ for all positive integers $n$ such that 
\begin{equation}\label{limit}
f(t_n)+f(t_n+1) \rightarrow 0 \text{ as } n \rightarrow \infty. 
\end{equation}
Since $f$ is assumed to be non-negative we have that both $f(t_n)$ and $f(t_n+1)$ vanish as $n \rightarrow \infty$. Notice that we can also say that
\begin{equation}\label{increasing}
t_n \uparrow \infty.
\end{equation}

Now recall that $f$ is non-negative and that there exists a positive real number $C$ such that $\int_0^t f(s)ds \leq C$ for all $t \in [0,\infty)$. Note as well that the intersection of $(2m,2m+2)$ with $(2n,2n+2)$ is empty if $m$ and $n$ are positive integers with the property that $m \neq n$. It follows that 
\begin{equation}\label{sum}
\int_{t_n}^{t_n+1} f(s) ds \rightarrow 0 \text{ as } n \rightarrow \infty.
\end{equation}
Combining \eqref{limit} with \eqref{increasing} and \eqref{sum}, we see that we have the lemma.
 
\end{proof}

\section{Proof of Theorem 1.2}

Proceeding as in \cite{GF} we argue as follows. First recall that we are assuming that a unique global pseudo classical solution $u$ to \eqref{pseudo} exists. Now, note as well that the hypotheses of Theorem 1.1 give us that
\begin{equation}
(F_1(u'(t)),u'(t))_{L^2} = c||u'(t)||_{L^2}^2 + d||u'(t)||_{L^3}^3,
\end{equation}
for all $t \in (0,\infty)$. We can now combine \eqref{eq 2.7} with Lemma 2.1 to see that there is a sequence $\{t_n\}_{n=1}^\infty$ such that $t_n \uparrow \infty$,
\begin{equation} \label{eq 3.1}
\lim_{n \rightarrow \infty} \int_{t_n}^{t_n+1} ||u'(s)||_{L^2}^2 ds =0,
\end{equation}
\begin{equation} \label{eq 3.2}
 \lim_{n \rightarrow \infty} \int_{t_n}^{t_n+1} ||u'(s)||_{L^3}^3 ds =0,
\end{equation}
\begin{equation} \label{eq 3.3a}
\lim_{n \rightarrow \infty} (||u'(t_n)||_{L^2} + ||u'(t_n+1)||_{L^2}) = 0,
\end{equation}
and
\begin{equation} \label{eq 3.3b}
\lim_{n \rightarrow \infty} (||u'(t_n)||_{L^3} + ||u'(t_n+1)||_{L^3}) = 0.
\end{equation}
Let $v \in L^2((a,b))$, and notice that we also have that
\begin{equation} \label{eq 3.4}
\begin{split}
|\int_{t_n}^{t_n+1} <u''(s),v> ds| & = |\int_{t_n}^{t_n+1} (u''(s),v)_{L^2} ds| \\ & =  |(u'(t_n+1),v)_{L^2} - (u'(t_n),v)_{L^2}| \\ & \leq |(u'(t_n+1),v)_{L^2}| + |(u'(t_n),v)_{L^2}| \\ & \leq ||u'(t_n+1)||_{L^2}||v||_{L^2} + ||u'(t_n)||_{L^2}||v||_{L^2}.
\end{split}
\end{equation}
Recalling \eqref{eq 3.3a}, we see that the limit as $n$ goes to $\infty$ of $\int_{t_n}^{t_n+1} <u''(s),v> ds$ is zero. 

Now, let us bound another integral of a quantity that shows up in global pseudo classical solutions of \eqref{pseudo}. Suppose $p$ is a real number greater than or equal to $2$ in value, suppose $q$ satisfies 
\begin{equation} \label{conjugate}
\frac{1}{q} + \frac{p-1}{p}=1,
\end{equation}
and note that
\begin{equation} \label{eq 3.5}
\begin{split}
& |\int_{t_n}^{t_n + 1} (|u'(s)|^{p-2}u'(s),v)_{L^2} ds| \\ & \leq \int_{t_n}^{t_n+1} ||u'(s)||_{L^p}^{p-1}||v||_{L^q} ds \\ & \leq (\int_{t_n}^{t_n+1} ||u'(s)||_{L^p}^p ds)^\frac{p-1}{p}(\int_{t_n}^{t_n+1} ||v||_{L^q}^q ds)^\frac{1}{q}. 
\end{split}
\end{equation}

We can now use \eqref{eq 3.1} and \eqref{eq 3.2} to conclude that the right hand side of the last inequality above vanishes as $n \rightarrow \infty$, provided that $p \in \{2,3\}$. We can now use the mean value theorem for integrals to conclude that for all $v \in H^2_*(\Omega)$ there exists a sequence $\{t_n^v\}_{n=1}^\infty$ with $t^v_n \in (t_n,t_n+1)$ for all $n$ such that 
\begin{equation} \label{eq 3.6}
\lim_{n \rightarrow \infty} [m<u''(t^v_n),v> + c(u'(t^v_n),v)_{L^2} + d(|u'(t^v_n)|u'(t^v_n),v)_{L^2}]=0.
\end{equation}

Now fix $v \in H^2_*(\Omega)$ and recall that $||u(t)||_{H^2_*}$ is bounded independent of $t$. It follows that there exists a function $\hat{u}_v \in H^2_*(\Omega)$ such that
\begin{equation} \label{eq 3.7}
u(t^v_n) \rightharpoonup \hat{u}_v \in H^2_*(\Omega)
\end{equation}
up to a subsequence. By compactness we have $u(t^v_n) \rightarrow \hat{u}_v$ in $L^2(\Omega)$. Now take a function $w \neq v$ and consider $\{t^w_n\}_{n=1}^\infty$. Then $u(t^w_n) \rightarrow \hat{u}_w$ in $L^2(\Omega)$, and $u(t^w_n) \rightharpoonup \hat{u}_w$ in $H^2_*(\Omega)$. By \eqref{eq 3.1}, H\"{o}lder's inequality, and the Fubini-Tonelli Theorem we obtain
\begin{equation} \label{eq 3.8}
\begin{split}
||u(t^v_n) - u(t^w_n)||_{L^2}^2 & =  \int_\Omega |\int_{t^v_n}^{t^w_n} u'(s) ds|^2 dx \\ & \leq \int_\Omega (\int_{t^v_n}^{t^w_n} |u'|^2 ds)(|t^w_n-t^v_n|)dx \\  & \leq |t^w_n - t^v_n| \int_{t^v_n}^{t^w_n} ||u'(s)||_{L^2}^2 ds \\ & \leq \int_{t_n}^{t_n + 1} ||u'(s)||_{L^2}^2 ds \rightarrow 0
\end{split}
\end{equation}
as $n \rightarrow \infty$. This shows that the limit $\hat{u}_v$ is independent of $v$, so let's call it $\hat{u}$. At this point, we have showed that for all $v \in H^2_*(\Omega)$ 
\begin{equation} \label{eq 3.9}
\begin{split}
& \sigma (\hat{u},v)_{H^2_*} + (F_2(\hat{u}),v)_{L^2} - (f,v)_{L^2} \\ & = \lim_{n \rightarrow \infty} [m<u''(t^v_n),v> + c (u'(t^v_n),v)_{L^2} + d(|u'(t^v_n)|u'(t^v_n),v)_{L^2} \\ & + \sigma (u(t^v_n),v)_{H^2_*} + (F_2(u(t^v_n)),v)_{L^2} - (f,v)_{L^2}] \\ & =0.
\end{split}
\end{equation}
This shows that $\hat{u}$ is the unique solution to \eqref{eq 1.1}. Now, subtracting the Euler-Lagrange equation contained in \eqref{eq 1.1} from the equation of motion contained in \eqref{pseudo}, we see that for all $v \in H^2_*(\Omega)$
\begin{equation} \label{eq 3.10a}
\begin{split}
& <u''(t),v> + c(u'(t),v)_{L^2} + d(|u'(t)|u'(t),v)_{L^2} + (u(t)-\hat{u},v)_{H^2_*} \\ & + (F_2(u(t))-F_2(\hat{u}),v)_{L^2} =0 \
\end{split}
\end{equation}
for all $t \in (0,\infty)$. This in turn allows us to write
\begin{equation} \label{eq 3.10b}
\begin{split}
& \int_{t_n}^{t_n+1} [<u''(s),u(s)-\hat{u}> + c(u'(s),u(s)-\hat{u})_{L^2} \\ & + d(|u'(s)|u'(s),u(s)-\hat{u})_{L^2} + (u(s)-\hat{u},u(s)-\hat{u})_{H^2_*} \\ & + (F_2(u(s))-F_2(\hat{u}),u(s)-\hat{u})_{L^2}] ds =0 
\end{split}
\end{equation} 
for all positive integers $n$.

Now, let us investigate the limiting behavior of the terms on the left hand side of the above equation. Let $p$ be a real number greater than or equal to $2$ in value, and let $\frac{1}{q} + \frac{p-1}{p}=1$, then we have
\begin{equation} \label{eq 3.11}
\begin{split}
|\int_{t_n}^{t_n+1} (|u'(s)|^{p-2}u'(s),u(s)-\hat{u})_{L^2} ds| & \leq \int_{t_n}^{t_n+1} ||u'(s)||_{L^p}^{p-1}||u(s)-\hat{u}||_{L^q}  ds\\  & \leq C \int_{t_n}^{t_n+1} ||u'(s)||_{L^p}^{p-1} ds \\ & \leq C(\int_{t_n}^{t_n+1} ||u'(s)||_{L^p}^p ds)^\frac{p-1}{p}.
\end{split}
\end{equation}
The quantity on the right hand side of the last inequality above vanishes as $n \rightarrow \infty$, if $p \in \{2,3\}$. Here, and for the remainder of the proof, $C$ will denote a constant that is independent of $n$ and that can change value from line to line. Note, we used \eqref{norm} to obtain a uniform bound on $||u(t)-\hat{u}||_{L^q}$. 

We also have that
\begin{equation}\label{eq 3.12}
\begin{split}
& |\int_{t_n}^{t_n+1} <u''(s),u(s)-\hat{u}> ds| \\ & = |\int_{t_n}^{t_n+1}(u''(s),u(s)-\hat{u})_{L^2}ds|\\
& \leq |\int_{t_n}^{t_n+1} (u'(s),u'(s))_{L^2} ds| + |(u'(t),u(t)-\hat{u})_{L^2}|_{t_n}^{t_n+1}| \\ & \leq \int_{t_n}^{t_n+1} ||u'(s)||_2^2 ds+ C(||u'(t_n + 1)||_{L^2} + ||u'(t_n)||_{L^2}) \\ & \rightarrow 0 
\end{split}
\end{equation}
as $n \rightarrow \infty$. Here we used \eqref{norm} to obtain a uniform bound on $||u(t)-\hat{u}||_{L^2}$.

Now, since $F_2$ is a non-decreasing function we also have that
\begin{equation} \label{eq 3.13}
(F_2(u(t))-F_2(\hat{u}), u(t)-\hat{u}) \geq 0
\end{equation}
for all $t \in (0,\infty)$. We can now use \eqref{eq 3.10b} to conclude that
\begin{equation} \label{eq 3.14}
\int_{t_n}^{t_n+1} ||u(s)-\hat{u}||_{H^2_*}^2  ds \rightarrow 0 \text{ as } n \rightarrow \infty.
\end{equation}

Now, recalling \eqref{eq 3.1}, we have
\begin{equation} \label{eq integral}
\int_{t_n}^{t_n+1} (||u(s)-\hat{u}||_{H^2_*}^2 + ||u'(s)||_{L^2}^2) ds \rightarrow 0 \text{ as } n \rightarrow \infty
\end{equation}  
and hence we can use the mean value theorem for integrals to conclude that there exists a $\bar{t}_n \in [t_n,t_n+1]$ for all positive integers $n$ such that 
\begin{equation} \label{eq 3.15}
||u(\bar{t}_n) - \hat{u}||_{H^2_*}^2 + ||u'(\bar{t}_n)||_{L^2}^2 \rightarrow 0 \text{ as } n \rightarrow \infty.
\end{equation}
This, in turn, implies that
\begin{equation} \label{eq 3.16}
\lim_{n \rightarrow \infty} E_u(\bar{t}_n) = \Sigma_T(\hat{u}) = I := \min_{v \in H^2_*(\Omega)} \Sigma_T(v).
\end{equation}
But $E_u(\bar{t}_n)$ is decreasing so that we have 
\begin{equation}\label{3.17}
\lim_{t \rightarrow \infty} E_u(t) = I.
\end{equation}

Now recall from section two that $E_u(t) = \frac{1}{2} ||u'(t)||_{L^2}^2 + \Sigma_T(u(t))$. From this it follows that $E_u(t) \geq \frac{1}{2} ||u'(t)||_{L^2}^2 + I$. Passing to the limit as $t \rightarrow \infty$ we see that
\begin{equation} \label{eq 3.18}
I = \lim_{t \rightarrow \infty} E_u(t)\geq I +  \limsup_{t \rightarrow \infty} \frac{1}{2} ||u'(t)||_{L^2}^2.
\end{equation}
The above allow us to conclude that $u'(t)  \rightarrow 0$ as $t \rightarrow \infty$. This, in turn, allows us to write
\begin{equation} \label{eq 3.19}
\lim_{t \rightarrow \infty} \Sigma_T(u(t)) = \lim_{t \rightarrow \infty} E_u(t) - \lim_{t \rightarrow \infty} \frac{1}{2} ||u'(t)||_{L^2}^2 = I = \Sigma_T(\hat{u}).
\end{equation}
It follows that $u(t) \rightarrow \hat{u}$ in $H^2_*(\Omega)$ on the whole flow. Theorem 1.2 follows.

\bibliographystyle{amsplain}

\begin{thebibliography}{9}

\bibitem {GF} A. Ferrero and F. Gazzola, A partially hinged rectangular plate as a model for suspension bridges, \textit{Discrete Contin. Dyn. Syst.} 35(12):5879-5908.

\bibitem {R} D. Raske, Hinged Beam Dynamics, arXiv:1912.04054v5  [math.AP].

\end{thebibliography}

\end{document}